\documentclass[11pt]{amsart}

\usepackage[
backend=biber,
style=alphabetic,
sorting=nyt
]{biblatex} 
\usepackage{enumerate}

\RequirePackage[utf8]{inputenc}
\usepackage{amsmath,amssymb}
\usepackage{amsfonts}
\usepackage{amsthm}
\usepackage{latexsym}
\usepackage{tikz}
\usetikzlibrary{shapes.geometric, arrows}
\usetikzlibrary{patterns}
\usepackage{graphicx}
\usepackage{epsfig}
\usepackage{enumerate}
\usepackage{xcolor}
\usepackage{hyperref}
\newtheorem{theorem}{Theorem}[section]
\newtheorem{lemma}[theorem]{Lemma}
\newtheorem{proposition}[theorem]{Proposition}
\newtheorem{definition}[theorem]{Definition}

\newtheorem{example}{Example}

\newtheorem{prop}{Proposition}[section]
\newtheorem{remark}[prop]{Remark}

\makeatletter \@addtoreset{equation}{section} \makeatother

\def\vol{\mathrm{vol}}
\def\av{\mathrm{Av}}
\def\dv{\mathrm{dV}}
\def\h{\mathrm{d}}
\def\div{\mathrm{div}}
\def\ric{\mathrm{Ric}}

\def\RR{{\mathrm R}}
\def \2R{{\hat{\RR}}}

\def\He{{\mathrm{Hess}}}

\def\xi{\partial_{x_i}}

\usepackage[margin=1.5 in]{geometry}

\begin{document}

\title{A lower bound for the first non-zero basic eigenvalue on a singular Riemannian foliation}
\author{Bach Tran}
\address{University of Oklahoma\newline
\indent Department of Mathematics\newline
\indent 601 Elm Ave\newline
\indent Norman, OK, 73019-3103, USA}
\email{bach.n.tran-1@ou.edu}
\date{\today}
\begin{abstract}
    In the paper, we provide a lower bound of the first non-zero basic eigenvalue on a closed singular Riemannian manifold $(M^n,\mathcal{F})$ having basic mean curvature that depends on the given lower bound of the Ricci curvature $\ric(M)\geq (n-1)K$ and the diameter of the leaf space $\h_{M/\mathcal{F}}$. For foliations given by points, the results cover estimates of the first non-zero eigenvalues lower bounds, which are Zhong-Yang estimate, Yang estimate and Shi-Yang estimate . 

We also show a rigidity result that corresponds to the case of $\ric(M)\geq 0$. This is  a generalization of Hang-Wang rigidity, which is the equality case of Zhong-Yang estimate. 
\end{abstract}
\maketitle
\tableofcontents
\markboth{Bach Tran}{A lower bound for the first non-zero basic eigenvalue on a singular Riemannian foliation}

\section{Introduction}\label{00000}
The study of how curvature influences the eigenvalues of the \textit{(ordinary) Laplacian} has been a central theme in Riemannian geometry since the 1950s. One interesting problem in that trend is giving estimates for lower and upper bounds of the first eigenvalue of the Laplacian of a compact Riemannian manifold under curvature assumptions. In this paper, specifically, we are interested in study the first eigenvalue under the constraint of a Ricci curvature lower bound.

Recall that the Laplacian, or the \textit{Laplace-Beltrami operator} on a Riemannian manifold $(M^n,g)$ is given by
$$\Delta=\div(\nabla).$$
A value $\lambda$ will be called an \textit{eigenvalue} of the Laplacian of $M$ if there exists a non-zero function $f\in C^\infty(M)$ (that we call an  eigenfunction with the eigenvalue $\lambda$) such that 
$$\Delta f=-\lambda f.$$
When $M$ is compact, we know the fact that $\lambda$ is non-negative.

Let $\lambda_1(M)$ be the first non-zero eigenvalue of the Laplacian on $M$ and $\lambda_1$ if we are clear about which manifold we work on.

In the case that $\ric(M)\geq (n-1)K\geq 0$ for some constant $K$, many results have been found: Lichnerowicz \cite{lich} and then Obata \cite{obata} proved that on a closed Riemannian manifold with $K>0$, one has $\lambda_1\geq nK$ and the equality holds if and only if $M$ is isometric to the $n$-sphere of radius $1/\sqrt{K}$. In the case $\ric(M)\geq 0$, we cannot get a positive lower bound without a restriction of the diameter $\h_M$ of $M$: for example, in the case of torus, if we increase the diameter arbitrarily and approaching infinity, then the first eigenvalue approaches $0$. Moreover, the diameter restriction is automatic in the case of $K>0$ by Bonnet-Myer's theorem. Hence, the diameter restriction is necessary. By Li and Yau \cite{liyau} and the improvement by Li \cite{li}, we get a gradient estimate for the first nontrivial eigenfunction and derive that
$$\lambda_1\geq\frac{\pi^2}{2\h_M^2}.$$

Zhong and Yang \cite{zhongyang} enhanced this result to be an optimal one for the case of non-negative Ricci curvature:
$$\lambda_1\geq \frac{\pi^2}{\h_M^2}.$$
Hang and Wang \cite{hangwang} later provided and proved the rigidity statemtent for Zhong-Yang estimate: the equality holds if and only if $M$ is isometric to a circle of radius $\frac{\h_M}{\pi}$. 

 Inspired by foundational results of Lichnerowicz and Zhong-Yang, Li had conjectured that the first positive eigenvalue should satisfy the following estimate (see \cite{dagang})
\begin{align*}
    \lambda_1\geq\frac{\pi^2}{\h_{M}^2}+(n-1)K.
\end{align*}
The conjecture greatly motivates many related studies in this area and an effort to prove the so-called Li's conjecture will unify Zhong-Yang's estimate and Lichnerowicz's estimate. Many people tried to prove this conjecture, and in particular, some towards improved inequalities of the ``Li's type'' form
\begin{align}\label{aha}
    \lambda_1\geq\frac{\pi^2}{\h_{M}^2}+\alpha(n-1)K.
\end{align} for some constant $\alpha$ were made. We can see there are many works about general optimal lower bound estimate for $\lambda_1$ for a given lower bound $(n-1)K$ of the Ricci curvature, for example, in \cite[etc.]{benjulie,benni,bakryqian,yuehe,ling1,ling,dagang,zhangwang}. Especially, a remarkable result by Shi and Zhang in \cite{shizhang} proved that for any $s\in (0,1)$
\begin{align*}
    \lambda_1\geq4s(1-s)\frac{\pi^2}{\h_M^2}+s(n-1)K.
\end{align*}
Notice that for $s=\frac{1}{2}$, we have a particular Li's type estimate
\begin{align*}
    \lambda_1\geq\frac{\pi^2}{\h_M^2}+\frac{1}{2}(n-1)K.
\end{align*}

B. Andrews and J. Clutterbuck \cite{benjulie} confirmed that the inequality of this form (\ref{aha}) with $\alpha=\frac{1}{2}$ is the best possible constant and this means Li's conjecture is false.

On the other case, when $\ric(M)\geq (n-1)K$ where $K<0$,  Li and Yau \cite{liyau} proved that
\begin{align}\label{liyauestimating}
    \lambda_1\geq\frac{\exp\left(-1-\sqrt{1+4(n-1)^2\h_{M}^2K}\right)}{2(n-1)\h_{M}^2}
\end{align}

However, if we consider the behavior of $K$ at $0$, namely, when $K\to 0$, then (\ref{liyauestimating}) gives us
\begin{align*}
    \lambda_1\geq\frac{\exp(-2)}{2(n-1)\h_M^2},
\end{align*}which does not match the Zhong-Yang estimate as mentioned above. Yau conjectured that there should be an estimate of the form
\begin{align*}
    \lambda_1\geq\frac{\pi^2}{\h_M^2}\exp\left(-C_n\sqrt{(1-n)K\h_M^2}\right)
\end{align*}for some positive constant $C_n$ depends only on the dimension $n$ of $M$.

Then Yang \cite{yangneg} proved this conjecture with the constant $C_n=\max\{\sqrt{n-1},\sqrt{2}\}$.

So how is the lower bound estimates of the first eigenvalue related to the object we will discuss in this paper, the first \textit{basic eigenvalue} on a \textit{singular Riemannian foliation}? We first introduce the concept of singular Riemannian foliations. A \textit{singular Riemannian foliation} on a Riemannian manifold $M$, roughly speaking, is a partition of $M$ into connected injectively immersed submanifolds (we call \textit{leaves}), which are locally equidistant to one another (see Subsect.~\ref{121212}). One class of examples of singular Riemannian foliations are orbits of isometric actions of Riemannian manifolds, that we call \textit{homogeneous}. Another class of examples are connected components of the fibers of Riemannian submersions, that we call \textit{simple}. For more detail about homogeneous and simple singular Riemannian foliations, see \cite{radeschi}. 

Given $(M,\mathcal{F})$ a singular Riemannian foliation, we define $\mathcal{F}$-\textit{basic function}, is a function that is constant on each leaf of $\mathcal{F}$. When the know which $\mathcal{F}$ we talk about, we can call it briefly \textit{basic function}. This allows us to define $\mathcal{F}$-\textit{basic Laplacian} as an operator that sends basic smooth functions to basic smooth functions and define $\mathcal{F}$-\textit{basic eigenvalues}, that we will discuss more detail in Sect.~ \ref{121}. When $(M,\mathcal{F})$ admits a \textit{basic mean curvature}, it turns out that the ordinary Laplacian coincides with the basic Laplacian, and any results that work for the ordinary Laplacian can be used to derive corresponding results for the basic Laplacian. In the case of estimate a lower bound for the first non-zero basic eigenvalue $\lambda_1^B$, if we use Lichnerowiz's estimate, we can get $\lambda_1^B\geq \lambda_1\geq nK$ but it does not imply any better result because $nK$ does not depend on $\mathcal{F}$ that we consider, and we do not see the difference between Lichnerowicz's estimate for the first non-zero eigenvalue and Lichnerowic'z estimate for the first non-zero basic eigenvalue. However, for the estimates involving the diameter like the results we mentioned earlier by Li-Yau \cite{liyau}, Li \cite{li}, Zhong-Yang \cite{zhongyang}, Shi-Yang \cite{shizhang}, etc., we can generalize them by replacing the diameter $\h_M$ of the ambient manifold $M$ with the diameter of the leaf space $\h_{M/\mathcal{F}}$ and it is reasonable to expect similar estimates with leaf space diameters. Hence, in this paper, we assume that the Ricci curvature of the ambient manifold is bounded below by a non-negative value. 

In the case of $K<0$, we can  can generalize Yang estimate for singular Riemannian foliations
\begin{theorem}\label{main3}
    Let $M^n$ be a compact Riemannian manifold with $\ric(M)\geq(n-1)K$, where $K<0$ and $\mathcal{F}$ is a singular Riemannian foliation on $M$ with closed leaves and basic mean curvature. Then we have a lower bound for the first non-zero basic eigenvalue $\lambda_1^B$ of the basic Laplacian given by 
    \begin{align*}
        \lambda_1^B\geq\frac{\pi^2}{\h_{M/\mathcal{F}}^2}\exp\left(-C_n\sqrt{(1-n)K\h_{M/\mathcal{F}}^2}\right),
    \end{align*}where $C_n=\max(\sqrt{n-1},\sqrt{2})$.
\end{theorem}

In the case of non-negative Ricci curvature, we obtain a generalized version of Zhong-Yang estimate for singular Riemannian foliations
        \begin{theorem}\label{main1}

    Let $M^n$ be a compact Riemannian manifold with non-negative Ricci curvature and $\mathcal{F}$ a singular Riemannian foliation of $M$ with closed leaves and basic mean curvature. Then we have a lower bound for the first non-zero basic eigenvalue $\lambda_1^B$ of the Laplacian, given by \begin{align*}
        \lambda_1^B\geq \frac{\pi^2}{\h_{M/\mathcal{F}}^2}.
    \end{align*}
\end{theorem}

We also have the rigidity of Theorem \ref{main1}, that generalize Hang-Wang result as the following
\begin{theorem}\label{main2}
  Let $M^n$ be a connected compact Riemannian manifold with non-negative Ricci curvature and $\mathcal{F}$ a singular Riemannian foliation of $M$ with closed leaves and basic mean curvature. Assume that $\displaystyle\lambda_1^B=\frac{\pi^2}{\h_{M/\mathcal{F}}^2}$. Then $\mathcal{F}$ is a regular Riemannian foliation and either
    \begin{enumerate}[(1)]
        \item $M$ is isometric to a mapping torus $([0,2\h_{M/\mathcal{F}}]\times N)/\sim $ with $(0,x)\sim(2\h_{M/\mathcal{F}},\varphi(x))$, where $N$ is a compact $(n-1)$-dim Riemannian manifold with $\ric\geq 0$ and $\varphi:N\to N$ is an isometry of $N$, and $\mathcal{F}$ is formed by leaves $\{(\{t\}\times N)/\sim\}_{t\in[0,2\h_{M/\mathcal{F}}]}$, or,

        \item $M$ is isometric to $([0,2\h_{M/\mathcal{F}}]\times N)/\sim $ with $(t,x)\sim(-t+2\h_{M/\mathcal{F}},x)$ and $(0,x)\sim(2\h_{M/\mathcal{F}},\varphi(x))$ (a quotient of a mapping torus with an involution map), where $N$ is a compact $(n-1)$-dim Riemannian manifold with $\ric\geq 0$ and $\varphi:N\to N$ is an isometry of $N$, and $\mathcal{F}$ is formed by leaves $\{(\{t\}\times N)/\sim\}_{t\in[0,2\h_{M/\mathcal{F}}]}$.

    \end{enumerate}
    In the first case,  $M/\mathcal{F}$ is isometric to the circle of radius $\displaystyle\frac{\h_{M/\mathcal{F}}}{\pi}$. In the second case,  $M/\mathcal{F}$ is isometric to the closed interval $[0,\h_{M/\mathcal{F}}]$.
   
\end{theorem}

Finally, we have the generalized version of Shi-Zhang estimate for singular Riemannian foliations when we have a non-negative lower bound of the Ricci curvature
\begin{theorem}
   \label{main}
    Let $M^n$ be a compact Riemannian manifold and $\mathcal{F}$ a singular Riemannian foliation of $M$ with closed leaves and basic mean curvature. Assume the Ricci curvature of $M$ is bounded below by $(n-1)K\geq 0$. Suppose $\lambda_1^B$ is the first non-zero basic eigenvalue of $(M,\mathcal{F})$. Then
    \begin{align}\label{uuuuuu}
        \lambda_1^B\geq 4s(1-s)\frac{\pi^2}{\h_{M/\mathcal{F}}^2}+s(n-1)K,
    \end{align}for any constant $s$ such that $0<s<1$.

\end{theorem}

Particularly, in the case the diameter of $M$ is small enough, we can choose a specific value of $s$ to get the optimal estimate:
    \begin{align*}\lambda_1^B\geq\frac{\pi^2}{\h_{M/\mathcal{F}}^2}+\frac{1}{2}(n-1)K+\frac{(n-1)^2K^2\h_{M/\mathcal{F}}^2}{16\pi^2}.
    \end{align*}

    We will discuss more after proving Theorem \ref{main} in Sect.~ \ref{1111}.

\begin{remark}
Notice that when $\mathcal{F}$ is a foliation by points (orbits by the isometric action of the identity element), $M/\mathcal{F}$ is the same as the ambient manifold $M$, and we recover Yang estimate, Zhong-Yang estimate, Hang-Wang rigidity, and Shi-Zhang estimate.
\end{remark}

The paper is structured as follows: in Sect.~ \ref{121}, we recall the main definitions and important results that we will use throughout the paper. In Sect.~ \ref{7777}, we prove Theorem \ref{main1} and Theorem \ref{main2}. In Sect. \ref{casenegative}, we prove Theorem \ref{main3}. Finally, in Sect.~ \ref{1111}, we prove Theorem \ref{main}.

\textbf{Acknowledgements.}
I would like to thank Dr. Ricardo Mendes, my Ph.D advisor, for many great ideas and suggestions of readings. His thoughtful feedback on my drafts, especially the discussion after the proof of Theorem \ref{main}, was greatly helpful in improving this paper. Also, I would like to thank Dr. Ken Richardson and Dr. Catherine Searle for comments on the last version of the rigidity statement in my last talk, it helped me figuring out how to state and prove the rigidity statement of Theorem \ref{main2} correctly. And thank Dr. Michael Jablonski and Dr. Marco Radeschi for comments and suggestion of writing on the latest version.

\section{Preliminaries}\label{121}
\subsection{Singular Riemannian Foliations}\label{121212}
\subsubsection{Basic definitions and facts}
Singular Riemannian foliations, first introduced in \cite{molino} by Molino in his study of Riemannian foliations and constitute nowadays an active field of research. 
\begin{definition}\label{667}
    Let $\mathcal{F}$ be a decomposition of a Riemannian manifold $(M,g)$ into connected injectively immersed submanifolds, called \textit{leaves}, which may have different dimensions. We say that $\mathcal{F}$ is a \textit{singular Riemannian foliation} if the following conditions are satisfied:
    \begin{enumerate}[(i)]
        \item $\mathcal{F}$ is a \textit{transnormal system}, that is, every geodesic orthogonal to one leaf remains orthogonal to all the leaves that it intersects,
        \item $\mathcal{F}$ is a \textit{singular foliation}, that is, $T_pL=\{X_p:X\in\mathcal{X}_{\mathcal{F}}\}$ for every leaf $L$ in $\mathcal{F}$ and every $p\in L$, where $\mathcal{X}_\mathcal{F}$ is the module of smooth vector fields on the ambient manifold $M$ that are everywhere tangent to the leaves of $\mathcal{F}$.
    \end{enumerate}

\end{definition}

The first condition in Definition \ref{667} is equivalent to the leaves being locally equidistant, see \cite[Proposition 2.3]{radeschi}. If every leaf of $\mathcal{F}$ is closed, we say $\mathcal{F}$ is \textit{closed}. In this case, the condition of local equidistance of leaves becomes the global equidistance of leaves.

For a given singular Riemannian foliation $(M,\mathcal{F})$, leaves form a space that we call the \textit{leaf space}, denoted by $M/\mathcal{F}$. More precisely, we define an equivalence relation on $M$ by saying that $x\sim y$ if $x$ and $y$ are in the same leaf $L_x\in\mathcal{F}$ containing $x$, and then the corresponding element of $M/\mathcal{F}$ is $[L_x]$. In addition, if the leaves are closed, then because of the global equidistance of leaves, the leaf space $M/\mathcal{F}$ can be equipped with a well-defined metric structure, that is, for every $L_p,L_q\in \mathcal{F}$, we define the metric $\rho$ by \begin{align*}
    \rho([L_p],[L_q]):=\h(L_p, L_q)=\h(p,L_q)=\h(L_p,q)=\inf_{z\in L_q}\h(p,z)=\inf_{t\in L_p}\h(t,q).
\end{align*} Then the natural projection map $\sigma:M\to M/\mathcal{F}$ where $M/\mathcal{F}$ is equipped the metric $\rho$ is a map so-called \textit{manifold submetry}. For more detail, see \cite{mendesradeschi}. We can define the diameter of the leaf space $M/\mathcal{F}$ by
\begin{align*}
    \h_{M/\mathcal{F}}=\sup_{L,L'\in \mathcal{F}}\h(L,L').
\end{align*}

We define a function on $M$ to be \textit{$\mathcal{F}$-basic} (or just \textit{basic}, if $\mathcal{F}$ is understood from context) when it is constant on the leaves of $\mathcal{F}$. We denote the set of smooth basic functions on $M$ by $C_B^\infty(M)$.

 The leaves of maximal dimension are called \textit{regular} and the other ones are \textit{singular}. The points of $M$ are said to be regular or singular according to the leaves through them. A singular Riemannian foliation is called \textit{regular} if all leaves are regular, that is, if it is a Riemannian foliation. The \textit{dimension} of $\mathcal{F}$ is the maximal dimension of the leaves and its \textit{codimension} is $\dim M-\dim\mathcal{F}$.
\subsubsection{Basic mean curvature}

Let $(M,\mathcal{F})$ be a singular Riemannian foliation. Then there exists an open, dense, saturated, full-measure subset $M_0\subset M$, which is called the \textit{principal part}, such that $M_0/\mathcal{F}\subset M/\mathcal{F}$ is a smooth Riemannian manifold, and the restriction of $\sigma$ on $M_0$, $\sigma|_{M_0}:M_0\to M_0/\mathcal{F}$ is a Riemannian submersion (see \cite[p. 3]{radeslyt}).

For every $p\in M_0$, denote $H(p)\in T_pM$ for the mean curvature vector of the leaf $L_p\subset M$. These mean curvature vectors form a smooth vector field $H$ on $M_0$. We say that $\mathcal{F}$ has a \textit{basic mean curvature} when $H$ is $\sigma$-related to a smooth vector field $\overline{H}$ on $M_0/\mathcal{F}$.
Having a basic mean curvature is automatic when considering a singular Riemannian foliation on homogeneous foliations; or if the ambient space is either round sphere, or $\mathbb{KP}^n$ (where $\mathbb{K}=\mathbb{R},\mathbb{C},\mathbb{H}$), see \cite{radesalexmeanflow}).

\subsection{Basic Laplacian and basic eigenvalues}

\subsubsection{Averaging operator}\label{7''}

Let $(M,\mathcal{F})$ be a singular Riemannian foliation with compact leaves. For a given $L^2$- function $f:M\to\mathbb{R}$, we can define its \textit{average }$\av(f):M\to\mathbb{R}$ over $\mathcal{F}$ by
\begin{align*}
    \av(f)(p)=\frac{1}{\vol(L_p)}\int_{ L_p}f(q)\dv^p(q),
\end{align*}where $\dv^p$ denotes the Riemannian volume form of the compact Riemannian manifold $L_p$ and $\vol(L_p)$ denotes its volume.

When $\mathcal{F}$ admits a basic mean curvature, the operator $\av$ takes smooth functions to smooth functions, and it commutes with the Laplacian on the principal part $N$ (see \cite[Lemma 3.2 and Theorem 3.3]{radeslyt}). In particular, see \cite{parken} for more detail for the case of regular Riemannian foliations.

\subsubsection{Basic Laplacian}
In \cite{parken}, we can see the definition of the basic Laplacian for Riemannian foliations. 

With the assumption of having basic mean curvature, the ordinary Laplacian coincides with the basic Laplacian: Given a singular Riemannian foliation $(M,\mathcal{F})$, where $(M,g)$ is a compact Riemannian manifold. Let $\Delta:=\div\nabla:C^\infty(M)\to C^\infty(M)$ be the (ordinary) Laplace operator on smooth functions over $M$. Recall from Subsect.~\ref{7''} that $\av$ commutes with $\Delta$ on $C^\infty(M)$. Therefore, the Laplacian $\Delta$ sends basic smooth functions to basic smooth functions and it turns out that $\Delta$ is exactly the basic Laplacian.

In this case, we call $\lambda$ a \textit{basic eigenvalue} if there is a non-zero basic smooth function $f\in C_B^\infty(M)$ such that $\Delta f=-\lambda f$ and $f$ is the corresponding \textit{basic eigenfunction} of $\lambda$.

\section{The case of Non-negative Ricci Curvature}\label{7777}

In this section, we study a lower bound for the case of singular Riemannian foliations on a compact Riemannian manifold with the assumption of \textbf{non-negative Ricci curvature}.

\subsection{A generalization of Zhong-Yang estimate.}\label{888}

We will prove a generalized Zhong-Yang estimate in the case of singular Riemannian foliation. 

In the setting, we assume $M^n$ is a compact Riemannian manifold  with non-negative Ricci curvature and $\mathcal{F}$ is a closed singular Riemannian foliation on $M$ that has a basic mean curvature.

We recall the gradient estimate by Zhong and Yang (\cite{zhongyang}) in Proposition \ref{oo}. Once notices that their argument works with not the first eigenvalue, but also with an arbitrary eigenvalue $\lambda>0$ in the spectrum.

Let $u\in C^\infty(M)$ be an eigenfunction with $\Delta u=-\lambda u$, where $\lambda>0$. By Green's theorem, we obtain that $\int_M{u \dv}=0$. Because $u$ is non-zero, this implies that $\max u>0>\min u$.

    If $\max u\geq -\min u$, then we can consider new function $v=u/\max u$ such that $\Delta v=-\lambda v$ and $\max v=1,-1\leq\min v<0$. Otherwise, we consider the new function $v=-u/\min u$, and then $\Delta v=-\lambda v$ and $\max v=1,-1\leq \min v<0$. Therefore, without loss of generality, we can assume that $u$ is an eigenfunction with respect to $\lambda$ such that
    $$1=\max u>\min u=-k,0<k\leq 1.$$

    Let $\varepsilon>0$ be an arbitrary small positive number and let 
    \begin{align}\label{uia}
        v_{\varepsilon}=\frac{u-\frac{1-k}{2}}{(1+\varepsilon)\frac{1+k}{2}}.
    \end{align}

    Then we see that $v_\varepsilon$ is a smooth function that satisfies
    \begin{align*}
        \left\{\begin{array}{c}
             \Delta v_\varepsilon=-\lambda(v_\varepsilon+a_\varepsilon) \\
             \displaystyle\max v_\varepsilon=\frac{1}{1+\varepsilon},\min v_\varepsilon=-\frac{1}{1+\varepsilon}
        \end{array}\right.
    \end{align*}where $a_\varepsilon=\frac{1-k}{1+k}\cdot\frac{1}{1+\varepsilon}$.

    Set $v_\varepsilon=\sin\theta_\varepsilon$, then the function $\theta_\varepsilon=\sin^{-1}v_\varepsilon$ has its range is $\left[\frac{-\pi}{2}+\delta,\frac{\pi}{2}-\delta\right]$, where $\delta$ is given by
    $$\sin\left(\frac{\pi}{2}-\delta\right)=\frac{1}{1+\varepsilon}.$$

\begin{proposition}\label{oo}
With the setting above, we have that \begin{align}\label{zzzzz}
    \vert\nabla\theta_\varepsilon\vert^2\leq \lambda(1+a_\varepsilon\psi(\theta_\varepsilon)),
\end{align}where $\psi:\left[-\frac{\pi}{2},\frac{\pi}{2}\right]\to\mathbb{R}$ is the function given by
\begin{align*}
        \left\{\begin{array}{lc}
             \psi(\theta)=\frac{4}{\pi}(\theta\sec^2\theta+\tan\theta)-2\tan\theta\cdot\sec\theta,& \theta\in\left(-\frac{\pi}{2},\frac{\pi}{2}\right) \\
             \psi\left(\frac{\pi}{2}\right)=1,\text{ }\psi\left(-\frac{\pi}{2}\right)=-1& 
        \end{array}\right.
    \end{align*}
\end{proposition}
\begin{remark}
   Notice that $\psi$ is an odd function which is continuous on $\left[-\frac{\pi}{2},\frac{\pi}{2}\right]$ and is the solution of the ODE
    \begin{align}\label{o}
        \psi''(\theta)-2\tan\theta\cdot\psi'(\theta)-2\sec^2\theta\cdot \psi(\theta)=-2\tan\theta\cdot\sec\theta.
    \end{align}
    
\end{remark}

\begin{proof}
See \cite{zhongyang}, or see also \cite{schoenyau}, \cite[p. 37-38]{xianguo}.

 \end{proof}   

 We are now ready to prove Theorem \ref{main1}. In the following proof, it is clear that when we have a gradient estimate, we can integrate over a geodesic joining two points of level sets of $u$ on both sides of the gradient estimate to derive an estimate for the first non-zero eigenvalue as many works in \cite{liyau,zhongyang} were done. The new thing in this proof is that we need to choose a more specific geodesic that works with the distances of leaves. This idea will be also demonstrated in proofs of Theorem \ref{main3} and Theorem \ref{main}:

        \begin{proof}[Proof of Theorem \ref{main1}]

        Let $u\in C^\infty_B(M)$ be a normalized basic eigenfunction corresponding to the eigenvalue $\lambda_1^B$, namely, a smooth basic function $u$ on $M$ satisfying $\Delta u=-\lambda_1^B u$ and $1=\max u>\min _{x\in M}u\geq -1$. 

        For any $\varepsilon>0$, let $\theta_\varepsilon=\sin^{-1}v_\varepsilon$ where $v_\varepsilon$ is defined as in (\ref{uia}). Notice that $\theta_\varepsilon$ is constant on each leaf of $\mathcal{F}$, namely, $\theta_\varepsilon$ is basic.
        
        Apply Proposition \ref{oo} to $\theta_\varepsilon$ and we get that
        \begin{align}\label{uiia}
            \vert\nabla\theta_\varepsilon\vert^2\leq\lambda_1^B(1+a_\varepsilon\psi(\theta_\varepsilon)),
        \end{align} and hence
        \begin{align}
            \label{aaaaa}\frac{\vert\nabla\theta_\varepsilon\vert}{\sqrt{1+a_\varepsilon\psi(\theta_\varepsilon)}}\leq\sqrt{\lambda_1^B}.
        \end{align}

        Let $x\in M$ such that $\theta_\varepsilon(x)=-\frac{\pi}{2}+\delta$.

        Choose $y$ from the level set $A=\theta_\varepsilon^{-1}(\frac{\pi}{2}-\delta)$ (which is compact and is a union of leaves of $\mathcal{F}$) minimizing the distance between $x$ and $A$ and let $\gamma
        $ be the geodesic joining $x$ and $y$. That means
        \begin{align*}
            L(\gamma)=\h(x,y)=\h(x,A).
        \end{align*}

        This geodesic $\gamma$ is also the shortest geodesic from $x$ to $L_y\subset A$, which means $d(x,A)=d(x,L_y)$.

        Since leaves of $\mathcal{F}$ are closed, the global equidistance of $\mathcal{F}$ indicates that $\h(x,L_y)=\h(L_x,L_y)$ and hence
        \begin{align*}
            L(\gamma)=\h(L_x,L_y)\leq\sup_{p,q\in M}\h(L_p,L_q)=\h_{M/\mathcal{F}}
        \end{align*}

        In (\ref{aaaaa}), take the integral along $\gamma$, we get that
        \begin{align}\label{sf}
            \h_{M/\mathcal{F}}\sqrt{\lambda_1^B}&\geq L(\gamma)\sqrt{\lambda_1^B}\nonumber\\&=\int_\gamma{\sqrt{\lambda_1^B }\h t}\geq\int_\gamma\frac{\vert\nabla\theta_\varepsilon\vert}{\sqrt{1+a_\varepsilon\psi(\theta_\varepsilon)}} \h t\geq\left\vert\int_\gamma{\frac{\nabla\theta_\varepsilon\cdot \h r}{\sqrt{1+a_\varepsilon\psi(\theta_\varepsilon)}}}\right\vert.
        \end{align}

        Because $\psi$ is an odd function, we see that
        \begin{align}\label{sg}
            \left\vert\int_\gamma{\frac{\nabla\theta_\varepsilon\cdot \h r}{\sqrt{1+a_\varepsilon\psi(\theta_\varepsilon)}}}\right\vert&=\int_{-\frac{\pi}{2}+\delta}^{\frac{\pi}{2}-\delta}{\frac{\h\theta}{\sqrt{1+a_\varepsilon\psi(\theta)}}}\nonumber\\&=\int_{0}^{\frac{\pi}{2}-\delta
            }\left(\frac{1}{\sqrt{1+a_\varepsilon\psi(\theta)}}+\frac{1}{\sqrt{1-a_\varepsilon\psi(\theta)}}\right)\h\theta\nonumber\\&=2\int_0^{\frac{\pi}{2}-\delta}\left(1+\sum_{k=1}^{\infty}{\frac{(4k-1)!!}{(4k)!!}a_\varepsilon^{2k}\psi^{2k}(\theta)}\right)\h\theta\\&\geq 2\left(\frac{\pi}{2}-\delta\right)+2\frac{3!!}{4!!}a_\varepsilon^2\int_0^{\frac{\pi}{2}-\delta}\psi^2(\theta)\h\theta\nonumber\\&=\pi-2\delta+\frac{3}{4}a_\varepsilon^2\int_0^{\frac{\pi}{2}-\delta}\psi^2(\theta)d\theta\nonumber.
        \end{align}where ``$!!$'' is for double factorial.
        
        Finally, by letting $\varepsilon\to 0$, which also makes $\delta\to 0$, (\ref{sf}) and (\ref{sg}) derive that 
        \begin{align}\label{uu}
            d_{M/\mathcal{F}}\sqrt{\lambda_1^B}\geq \pi+\frac{3}{4}a^2\int_0^{\frac{\pi}{2}}\psi^2(\theta)\h\theta,
        \end{align}where $a=\frac{1-k}{1+k}$.
        
        Therefore,
        \begin{align*}
            \lambda_1^B\geq\frac{\pi^2}{d_{M/\mathcal{F}}^2}.
        \end{align*}
\end{proof}

\begin{example}
    Let $M=S^1\times S^2$. Define $\mathcal{F}$ to be a collection of leaves of the form $L_{p}=S^1\times \{p\}$ where $p\in S^2$. This defines a simple foliation, whose leaves are fibers of the Riemannian submersion $\pi:S^1\times S^2\to S^2,(q,p)\mapsto p $. According to \cite[Section 2.2]{walschap}, $\mathcal{F}$ has a basic mean curvature.

    Observe that $\lambda_1^B=2$ with the corresponding first basic eigenfunction $f(q,p)=u(p)$ where $u$ is the first eigenfunction on $S^2$. The diameter of the leaf space is $\h_{M/\mathcal{F}}=\h_{S^2}=\pi$. Hence, we see that
    \begin{align*}
        \lambda_1^B=2>\frac{\pi^2}{\h_{M/\mathcal{F}}^2}=1,
    \end{align*}which verifies Theorem \ref{main1}.
\end{example}
\begin{example}\label{example2}
    Let $T=S^1\times S^1$ be a 2-torus and consider an isometric action of $S^1\cong\{e^{i\theta}\vert\theta\in[0,2\pi)\}\subseteq\mathbb{C}$ defined by $e^{i\theta}\cdot(x,y)=( x,e^{i\theta}\cdot y)$ where $(x,y)\in S^1\times S^1$ (here we regard $S^1=\mathbb{R}/(2\pi\mathbb{Z})$ and $x:=x+2\pi\mathbb{Z}$). 
    
    The orbits of this action are of the form $\{\{p\}\times S^1\}_{p\in S^1} $, and they form a (regular) Riemannian foliation $\mathcal{F}$. This is a homogeneous foliation that admits a basic mean curvature. The first basic eigenvalue is $\lambda_1^B=1$ with the corresponding eigenfunction is $f(x,y)=\sin(x)$. In this case, $T/\mathcal{F}= S^1$ and so $\h_{T/\mathcal{F}}=\pi$. It is obvious that \begin{align*}
        \lambda_1^B=\frac{\pi^2}{\h_{T/\mathcal{F}}^2}
    \end{align*}and this verifies the statement of Theorem \ref{main1}. 
    
\end{example}
\begin{example}\label{example3}
    Let $T=S^1\times S^1$ be a $2$-torus and consider an isometric action of $S^1\cong\{e^{i\theta}\vert\theta\in[0,2\pi)\}\subseteq\mathbb{C}$ defined by the diagonal action $e^{i\theta}\cdot(x,y)=(e^{i\theta}\cdot x,e^{i\theta}\cdot y)$ where $(x,y)\in S^1\times S^1$. 
    
    The orbits of this action are circles of the form $$\{L_{(p,q)}=\{e^{i\theta}(p,q)\vert\theta\in[0,2\pi)\}\}_{(p,q)\in S^1\times S^1}, $$ and they form a Riemannian foliation $\mathcal{F'}$. This is a homogeneous foliation that has a basic mean curvature. The first basic eigenvalue is $\lambda_1^B=2$ with the corresponding eigenfunction is $f(x,y)=\cos(x-y)=\cos x\cos y+\sin x\sin y$. In this case, $T/\mathcal{F}$ is isometric to  a circle of radius $\frac{1}{\sqrt{2}}$ and so $\h_{T/\mathcal{F}}=\frac{\pi}{\sqrt{2}}$. We can see that \begin{align*}
        \lambda_1^B=\frac{\pi^2}{\h_{T/\mathcal{F'}}^2}
    \end{align*}and this verifies the statement of Theorem \ref{main1}. 
\end{example}

\subsection{The rigidity of generalized Zhong-Yang estimate.}\label{ililil}

In both Example \ref{example2} and Example \ref{example3}, we get the equality case of Theorem \ref{main1}. This is not a coincidence because both singular Riemannian foliations are constructed from mapping tori. More precisely, we can see
\begin{itemize}
    \item In Example \ref{example2}, $T$ is isometric to $[0,2\pi]\times S^1/\sim$ where $(0,s)\sim (2\pi,s)$ for all $s\in S^1$  with the product metric. $\mathcal{F}$ is of the form $\{[\{t\}\times S^1]\vert t\in [0,2\pi]\} $.
    \item In Example \ref{example3}, $T$ is isometric to $[0,\pi\sqrt{2}]\times S^1/\sim$ where $(0,s)\sim (\pi\sqrt{2},s+\pi)$ for all $s\in S^1$ with the metric $dt^2+2g_{S^1}$ ($g_{S^1}$ is the round metric of $S^1$). $\mathcal{F}'$ is of the form $\{[\{t\}\times S^1]\vert t\in [0,\pi\sqrt{2}]\}$.
\end{itemize}

These examples belongs to the first class of Theorem \ref{main2}, and in this section, we will prove Theorem \ref{main2} which determine whether we have the equality $\lambda_1^B=\frac{\pi^2}{\h_{M/\mathcal{F}}^2}$ of Theorem \ref{main1}. 

For simplicity, we assume that $M$ is connected.

By imitating the argument of Hang-Wang \cite{hangwang} for the rigidity of Zhong-Yang's estimate with some modifications in taking geodesics joining points, we will get the following proposition
\begin{proposition}\label{8}
     Given the assumption as in the Theorem \ref{main1}. Assume that $\lambda_1^B=\frac{\pi^2}{\h_{M/\mathcal{F}}^2}$. Then $\mathcal{F}
     $ is of codimension one.
\end{proposition}

\begin{proof}

Denote $\lambda:=\lambda_1^B$. When $\lambda=\frac{\pi^2}{\h_{M/\mathcal{F}}^2}$, we see that in (\ref{uu}), we must have $a=0$.

The fact that $a=0$ implies $k=1$ and hence $\min u=-1$ and $\max u=1$.

First, we will prove the following lemma using precisely the argument for the proof of \cite[Lemma 1]{hangwang}:
\begin{lemma}\label{lmm1}
    Consider \begin{align}\label{1}
    \varphi=\vert\nabla u\vert^2+\lambda u^2.
\end{align}
Then we see that $\varphi=\vert\nabla u\vert^2+\lambda u^2\leq \lambda$.
\end{lemma}
\begin{proof}[Proof of Lemma \ref{lmm1}]
Choose a local orthonormal frame $\{E_i\}_{i=1}^n$ on $M$ and let $E_if=f_i$ for any $f\in C^\infty(M)$. Then from (\ref{1}), we have
\begin{align*}
    \varphi_i=E_i\varphi=E_i\langle\nabla u,\nabla u\rangle+E_i(\lambda u^2)=2\langle\nabla_{E_i}\nabla u,\nabla u\rangle+2\lambda uE_iu=2\sum_j{u_{j}u_{ji}}+2\lambda uu_i.
\end{align*}

Then, we see that
\begin{align*}
    \left\vert\frac{1}{2}\nabla\varphi-\lambda u\nabla u\right\vert^2=\sum_i\left(\sum_j{u_ju_{ji}}\right)^2\leq \vert \He u\vert^2\vert\nabla u\vert^2.
\end{align*}

Expand the left hand side and rearrange terms on both sides of the inequality to get
\begin{align*}
    \frac{1}{4}\vert\nabla\varphi\vert^2-\lambda  u\langle\nabla u,\nabla\varphi\rangle\leq\vert\nabla u\vert^2(\vert \He u\vert^2-\lambda^2 u^2).
\end{align*}

Therefore, on $\Omega=\{p\in M:\nabla u(p)\neq 0\}$, we have
\begin{align}\label{10}
    \vert \He u\vert^2-\lambda^2 u^2\geq\frac{\vert\nabla\varphi\vert^2-4\lambda u\langle\nabla u,\nabla\varphi\rangle}{4\vert\nabla u\vert^2}=\frac{\langle\nabla(\varphi-2\lambda u^2),\nabla\varphi\rangle}{4\vert\nabla u\vert^2}.
\end{align}

Use the Bochner's formula, we get
\begin{align}\label{from13}
    \frac{1}{2}\Delta\varphi=\frac{1}{2}\Delta(\vert\nabla u\vert^2+\lambda u^2)&=\vert \He u\vert^2+\langle\nabla u,\nabla\Delta u\rangle+\ric(\nabla u,\nabla u\rangle+\vert\nabla u\vert^2+\lambda u\Delta u\nonumber\\
    &=\vert \He u\vert^2-\vert\nabla u\vert^2+\ric(\nabla u,\nabla u\rangle+\vert \nabla u\vert^2-\lambda^2 u^2\nonumber\\&=\vert \He u\vert^2-\lambda^2 u^2+\ric(\nabla u,\nabla u)\\&\geq \frac{\langle \nabla(\varphi-2\lambda u^2),\nabla\varphi\rangle}{4\vert\nabla u\vert^2},\nonumber
\end{align}and thus,
\begin{align}\label{9}
    \Delta \varphi-\frac{\langle \nabla(\varphi-2\lambda u^2),\nabla\varphi\rangle}{2\vert\nabla u\vert^2}\geq 0.
\end{align}

By the maximum principle, we get that
\begin{align*}
    \varphi=\vert\nabla u\vert^2+\lambda u^2\leq\max_{\{\nabla u=0\}}(\vert\nabla u\vert^2+\lambda u^2)=\lambda.
\end{align*}

\end{proof}

Next, we will imitate the argument in \cite[p.4]{hangwang} to determine explicitly the formula of $u$ on a horizontal geodesic that realizes the leaf diameter: Let $x$ be a point such that $u(x)=-1$ and let $\gamma:[0,l]\to M$ be a (unit speed) horizontal geodesic that minimizes the distance from $x$ to the leaf $L_y$ of some point $y=\gamma(l)$ such that $u(y)=1$. Denote $f(t)=u(\gamma(t))$, and we get that
\begin{align}\label{3}
    \vert f'(t)\vert=\vert\langle \nabla u(\gamma(t)),\gamma'(t)\rangle\vert&\leq\vert\nabla u(\gamma(t))\vert\vert\gamma'(t)\vert=\vert\nabla u(\gamma(t))\vert\nonumber\\&\leq\sqrt{\lambda(1-[u(\gamma(t))]^2)}=\sqrt{\lambda(1-(f(t))^2)}.
\end{align}

    Now, for any $\varepsilon>0$, we have
    \begin{align}\label{2}
        \frac{\pi}{\sqrt{\lambda}}\geq l\geq\int_{\{t\in [0,l],f'(t)>0\}}\h t\geq\int_{0}^l{\frac{f'(t)\h t}{\sqrt{\lambda(1-(f(t))^2)}+\varepsilon}}=\int_{-1}^1{\frac
        {1}{\sqrt{\lambda(1-x^2)}+\varepsilon}dx}.
    \end{align}

    Let $\varepsilon\to 0^+$ and we get that
    \begin{align*}
        \frac{\pi}{\sqrt{\lambda}}\geq l\geq \int_{\{t\in [0,l],f'(t)>0\}}\h t\geq\int_{-1}^1\frac{1}{\sqrt{\lambda(1-x^2)}}dx =\frac{1}{\sqrt{\lambda}}\sin^{-1}x\bigg|_{-1}^{1}=\frac{\pi}{\sqrt{\lambda}}.
    \end{align*}
    
    This implies that $l=\frac{\pi}{\sqrt{\lambda}}$, and hence, the distance from $L_x$ to any leaf $L_z\subset u^{-1}(\{1\})$ (which is exactly the distance from $x$ to $L_z$ by the global equidistant) is always equal to $\frac{\pi}{\sqrt{\lambda}}$. Moreover, $f'(t)>0$ for almost everywhere $t\in (0,\pi)$. Hence, by The Fundamental Theorem of Calculus, we have that for any $x,y\in [0,\pi]$ such that $x>y$
    \begin{align*}
        f(x)=\int_y^x{f'(t)dt}+f(y)>f(y),
    \end{align*}which means $f$ is strictly increasing on $[0,\pi]$. Then from (\ref{2}) we get that
    \begin{align*}
        \int_0^{\pi/\sqrt{\lambda}}{\frac{f'(t)}{\sqrt{\lambda(1-(f(t))^2)}}dt}=\frac{\pi}{\sqrt{\lambda}},
    \end{align*}and from (\ref{3}) we imply that
    \begin{align*}
        f'(t)=\sqrt{\lambda(1-(f(t))^2)},
    \end{align*}for all $t\in (0,\pi/\sqrt{\lambda})$. By the continuity of $f$, we can see that 
    \begin{align}\label{4}
        f'(t)=\sqrt{\lambda(1-(f(t))^2)},
    \end{align}for all $t\in [0,\pi/\sqrt{\lambda}]$.

    Hence, $\lambda(f(t))^2+(f'(t))^2=\lambda$. Differentiate the equation with respect to $t$, we get \begin{align}\label{5}
        2\lambda f(t)f'(t)+2f'(t)f''(t)=0
    \end{align} From (\ref{4}), we know that $f'(t)>0$ for $t\in (0,\pi/\sqrt{\lambda})$. Then, we imply from (\ref{5}), we have $f''(t)+\lambda f(t)=0$. Notice that $f(0)=-1$ and $f'(0)=0$, we solve for $f$ and get that
    \begin{align}\label{formula}
    f(t)=u(\gamma(t))=-\cos(\sqrt{\lambda} t).
    \end{align}for all $t\in [0,\pi/\sqrt{\lambda}]$. This implies that $\He u(\gamma'(0),\gamma'(0))=\frac{\h}{\h t}\bigg|_{t=0}\langle \nabla u(\gamma(t)),\gamma'(t)\rangle=f''(0)=\lambda$. Since $\Delta u(x)=-\lambda u(x)=\lambda$ and $\He u(x)\geq 0$ (in the meaning of positive definite), we derive that
    \begin{align}\label{7}
        \He u(x)=\lambda(\gamma'(0))^\flat\otimes (\gamma'(0))^\flat,
    \end{align}where $(\gamma'(0))^\flat$ denote the cotangent vector given by $(\gamma'(0))^\flat(Y)=\langle \gamma
    '(0),Y\rangle$ for any $Y\in T_xM$. 

    We will need to use the following lemma, which is similar to \cite[Lemma 2]{hangwang} but for leaves:
    \begin{lemma}\label{6}
        There are at most two leaves $L_x,L_z$ where $u(x)=u(z)=-1$ and at most two leaves $L_y,L_w$ where $u(y)=u(w)=1$.
    \end{lemma}
    \begin{proof}[Proof of Lemma \ref{6}]

    Let $y$ be any point in $u^{-1}(\{1\})$ such that there is a unit speed horizontal geodesic $\alpha:[0,l']\to M$ that minimizes the distance from $x$ to the leaf $L_y$ at $y=\alpha(l')$. The same argument as above implies that $l'=\pi/\sqrt{\lambda}$ and 
    \begin{align}
        \He u(x)=\lambda(\alpha'(0))^\flat\otimes(\alpha'(0))^\flat.
    \end{align}
        From (\ref{7}), we get that $\vert\langle \gamma'(0),\alpha'(0)\rangle\vert=1$ and then $\alpha'(0)=\pm\gamma'(0)$. Thus  $y=\exp_x\left(\frac{\pi}{\sqrt{\lambda}}\gamma'(0)\right)$ or $y=\exp_x\left(-\frac{\pi}{\sqrt{\lambda}}\gamma'(0)\right)$, and we have at most two choices of $y$.

        Similar argument will be applied for the case of $z\in u^{-1}(\{-1\})$ to get at most two choices of geodesics minimizing the distances from $y\in u^{-1}(\{1\})$ to the leaf $L_z\subset u^{-1}(\{-1\})$ and therefore, we get the conclusion of the lemma.

    \end{proof}
    Now, we continue the proof of Proposition \ref{8} by using a contradiction with the assumption of the dimension argument, which is similar to \cite[p.5]{hangwang}: Assume the dimension of $M/\mathcal{F}$ is at least $2$. Let $N=M\setminus\{L_r:u(r)=\pm 1\}$. Then we see that $N/\mathcal{F}$ is still connected because we know from the Lemma \ref{6} that $\{L_r:u(r)=\pm 1\}$ has at most four leaves.

    We want to show that $\phi=\vert\nabla u\vert^2+\lambda u^2=\lambda$ on $N$. Indeed, let 
    \begin{align*}
        A=\{r\in N:\vert \nabla u(r)\vert^2+ \lambda(u(r))^2=\lambda\}.
    \end{align*}

    We see that $A$ is non empty because for any $t\in (0,\pi/\sqrt{\lambda})$
    \begin{align*}
        \vert\nabla u(\gamma(t))\vert^2+\lambda(u(\gamma(t))^2\geq\vert\langle\nabla u(\gamma(t)),\gamma'(t)\rangle\vert^2+\lambda(f(t))^2=(f'(t))^2+\lambda(f(t))^2=\lambda.
    \end{align*}
    
    Hence, we imply that $\vert\nabla u(\gamma(t))\vert^2+\lambda(u(\gamma(t))^2=\lambda$. Observe that $A$ is a closed set obviously. On the other hand, for any $r\in A\subset \Omega$, by (\ref{9}) and the strong maximum principle, we have $\vert\nabla u(q)\vert^2+\lambda (u(q))^2=\lambda$ for any $q$ in a small neighborhood of $r$. Therefore, the connectedness of $N$ implies that $E=N$.

    Define $X=\frac{\nabla u}{\vert\nabla u\vert}$ on $N$. 
    
    Since $\vert\nabla u\vert^2=\lambda(1-u^2)$ on $N$ so we imply that $X=\nabla\left(\frac{1}{\sqrt{\lambda}}\arcsin u\right)=\nabla g$. With the Ricci curvature assumption, the fact $\vert \nabla g\vert=1$ on $N$ implies that any integral curves of $X=\nabla g$ is a unit speed geodesic.

    Consider $\Sigma=\{r\in M:u(r)=0\}\subset N$ as a union of leaves. This is a hypersurface, and hence $\Sigma/\mathcal{F}$ has the dimension at least $1$.

    For an arbitrary $p\in \Sigma$, let $\beta_p$ be the maximal integral curve of $-X$ with the starting point $\beta_p(0)=p$. Then from the argument above we know that $\beta_p $ is a unit speed horizontal geodesic. Let $g_p(t)=u(\beta_p(t))$, then we see that $g_p(0)=0$ and \begin{align*}
        g'_p(t)=\langle\nabla u(\beta_p(t)),\beta_p'(t)\rangle&=\langle\nabla u(\beta_p(t)),-X(\beta_p(t))\rangle\\&=-\vert\nabla u(\beta_p(t))\vert\\&=-\sqrt{\lambda(1-u(\beta_p(t))^2)}\\&=-\sqrt{\lambda(1-(g_p(t))^2)}.
    \end{align*}
    Solve this ODE on $\left[0,\frac{\pi}{2}\right)$ with the initial condition is $g_p(0)=0$ and we get that $g_p(t)=-\sin t$ for $t\in [0,\frac{\pi}{2})$. On the other hand, $\beta_p$ is a geodesic in $M$, and it is defined on $[0,\infty)$. Hence, we get $g_p(t)=-\sin t$ for $t\in[0,\frac{\pi}{2}]$ and, in particular, $g_p(\frac{\pi}{2})=u(\beta_p(\frac{\pi}{2}))=-1$. Denote $s=\beta_p\left(\frac{\pi}{2}\right)$ and use the same argument as as in the Lemma \ref{6}, we have
    \begin{align*}
        \He u(s)=\He u \left(\beta_p\left(\frac{\pi}{2}\right)\right)=\lambda\left(\beta'_p\left(\frac{\pi}{2}\right)\right)^\flat\otimes\left(\beta_p'\left(\frac{\pi}{2}\right)\right)^\flat.
    \end{align*}

    We see that $p=\exp_s\left(-\frac{\pi}{2}\beta_p'\left(\frac{\pi}{2}\right)\right)$. According to Lemma \ref{6}, there are at most two leaves in the set $u^{-1}(\{1\})$ according to the Lemma \ref{6}, so we can find $L_q\subset u^{-1}(\{-1\})$ and infinitely many $L_p\in \Sigma$ (because $\dim \Sigma/\mathcal{F}\geq 1)$ such that $\beta_p (\frac{\pi}{2})\in L_q$, or more precisely, we can draw infinitely many horizontal geodesics joining infinitely many leaves in $\Sigma$ and one leaf in $u^{-1}(\{-1\})$. It is a contradiction, since $\beta'_p(\frac{\pi}{2})$ has at most two choices.

    Therefore, $M/\mathcal{F}$ must be of one dimension.
\end{proof}

Since $(M,\mathcal{F})$ is of codimension one, then either $M/\mathcal{F}$ is isometric to a circle or a closed interval. Surprisingly, we have no singular leaf in both cases, due to the following proposition:
\begin{proposition}
    Given the assumption as in Proposition \ref{8}. Then $\mathcal{F}$ has no singular leaf. 
\end{proposition}

\begin{proof}
    We start with the first case: Assume $M/\mathcal{F}$ is isometric to a circle. Let $x$ be a regular point of $(M,\mathcal{F})$. Then there exists a unit speed horizontal geodesic $\gamma:[a,b]\to M$ such that $\gamma(0)=x$ and $\gamma(t)$ is regular for every $t\in (a,b)$, where $a<0<b,b-a\leq\h_{M/\mathcal{F}}$, and $\gamma(0)=x$. The compactness of $M$ implies that $\gamma$ can be extended to $(-\infty,\infty)$.

 We will prove that $\gamma(b)$ is a regular point. Indeed, choose $n_0\in\mathbb{N}$ such that $0<b-\frac{1}{n_0}$. Since $\gamma$ is a horizontal minimizing geodesic, then on $\left[b-\frac{1}{n_0},b\right]$, the curve $\Tilde{\gamma}=\pi\circ\gamma$ is a minimizing geodesic. Since $M/\mathcal{F}$ is isometric to a circle, we can extend $\Tilde{\gamma}$ to the interval $\left[b-\frac{1}{n_0},b+\frac{1}{n_0}\right]$ which still realizes the distances between points on that geodesic because of $2/n_0\leq \pi$. The leaf $L_{\gamma(b-1/n_0)}$ is of codimension one and hence the unique horizontal lift of $\Tilde{\gamma}$ coincides with $\gamma$ and we know that $\gamma$ realizes the distances on $\left[b-\frac{1}{n_0},b+\frac{1}{n_0}\right]$.

 Apply Foliated Kleiner's lemma \cite[Lemma 3.24]{radeschi}, we see that $\dim L_{\gamma(b)}=\dim L_{\gamma(t)}$ for all $t\in \left(b-\frac{1}{n_0},b+\frac{1}{n_0}\right)$. Since $\gamma(t)$ is regular for any $t\in \left(b-\frac{1}{n_0},b\right]$, it implies that $\gamma(b)$ is also regular. 

 Repeat the procedure to extend the regularity through the entire $\gamma$ (which intersects all leaves of $\mathcal{F}$), and therefore, we conclude that $\mathcal{F}$ is a regular Riemannian foliation on $M$ for the first case.

 For the second case, assume that $M/\mathcal{F}$ is isometric to $[0,\h]$ where $\h=\h_{M/\mathcal{F}}$. We can identify $M/\mathcal{F}$ with $[0,\h]$.

    There is a unit speed horizontal geodesic $\gamma:[0,\h]\to M$ that intersects every leaves. Start with a regular point on $\gamma$ and apply the same argument as the previous case to have all $\gamma(t),t\in(0,\h)$, are regular points.

    It remains to show that $L_{\gamma(0)}$ and $L_{\gamma(\h)}$ are not singular leaves. Assume it not, then the principal part of $M_0=\sigma^{-1}((0,\h))$. Recall from Proposition \ref{8} that $f(t)=u(\gamma(t))=-\cos(\sqrt{\lambda}t)$, where $\lambda=\frac{\pi^2}{\h^2}$.

 Apply the formula at \cite[p.53]{walschap} for the principal part, we get that
    \begin{align*}
        -\lambda f(t)=-\lambda u(\gamma(t))=\Delta u(\gamma(t))=f''(t)-\langle H(\gamma(t)),\nabla u(\gamma(t))\rangle=f''(t)-H(t)f'(t),
    \end{align*}where $H(t)$ is the mean curvature of the leaf $L_{\gamma(t)}$. Use the fact $f(t)=-\cos(\sqrt{\lambda}t)$, we have
    \begin{align*}
        \lambda \cos(\sqrt{\lambda}t)=\lambda\cos(\sqrt{\lambda}t)-H(t)\sqrt{\lambda}\sin(\sqrt{\lambda}t).
    \end{align*}

    It implies that $H(t)=0$ for all $t\in (0,\h)$. We get a contradiction because we know the fact that the mean curvature blows up at the singularities, see \cite{radesalexmeanflow}.

    Therefore, $\mathcal{F}$ is a regular Riemannian foliation on $M$ for the second case.
\end{proof}

Finally, we can prove Theorem \ref{main2} with all gathered information:

\begin{proof}[Proof of Theorem \ref{main2}] According to the classification of Riemannian foliations of codimension one (see \cite[p.172 - 173]{molino}), we have two cases:

  \textbf{If $M/\mathcal{F}$ is isometric to a circle of radius $\frac{\h}{\pi}$:} $\mathcal{F}$ is transversally oriented and $\sigma:M\to M/\mathcal{F}$ is a Riemannian submersion with leaves are fibers. Let $\overline{\alpha}:\mathbb{R}\to M$ be the unique unit speed horizontal lift of the geodesic $\alpha:\mathbb{R}\to M/\mathcal{F}$ that parametrizes the circle $M/\mathcal{F}$ with $\overline{\alpha}(\h/2)\in u^{-1}(0)$. According to (\ref{formula}) in the proof of  Proposition \ref{8}, $u(\overline{\alpha}(t))=-\cos(\sqrt{\lambda}t)$, and hence, due to the formula in \cite[p.53]{walschap}, we see that
  \begin{align*}
        \lambda\cos(\sqrt{\lambda} t)=-\lambda u(\overline{\alpha}(t))=\Delta u(\overline{\alpha}(t))&=\frac{\h^2 }{\h t^2}u(\overline{\alpha}(t))-H(t)\frac{\h}{\h t}u(\overline{\alpha}(t))\\&=\lambda\cos(\sqrt{\lambda}t)-H(t)\sqrt{\lambda}\sin(\sqrt{\lambda}t),
    \end{align*}where $H(t)$ is the mean curvature of the leaf $L_{\overline{\alpha}(t)}$. 

    This implies that $H(t)=0$ for all $t$ and hence all leaves (which are fibers of $\pi$) are minimal.

    Then, a simple argument using Riccati equation of \cite[Proposition 5.7]{jacobitondeur} tells us that because $\mathcal{F}$ is of codimension one and all fibers are minimal with nonnegative Ricci curvature, then all fibers are totally geodesic.

From \cite[Lemma 1.4.3]{walschap}, the holonomy diffeomorphisms between fibers are isometries. Let $X$ be a unit vector field of $M/\mathcal{F}$ and consider $\overline{X}$ to be the horizontal lift of $X$ via $\sigma$. Then, $\overline{\alpha}$ is the integral curve of $\overline{X}$ with the initial point $\overline{\alpha}(0)$. Follow the flow $\psi_t:M\to M$ generated by $\overline{X}$, we obtain an isometry $\varphi:L_0:=L_{\overline{\alpha}(0)}\to L_0$ such that $x\mapsto \psi_{2\h}(x)$. Therefore, $M=[0,2\h]\times_\varphi N$ is isometric to a mapping torus of an isometry $\varphi:N\to N$ with a $(n-1)$-dimensional compact connected Riemannian manifold $N$ of nonnegative Ricci curvature and $\mathcal{F}$ is formed by $\{L_t=\{t\}\times N/\sim\}_{t\in[0,2\h]}$.

\textbf{If $M/\mathcal{F}$ is isometric to a closed interval $[0,\h]$:} $\mathcal{F}$ is not transversally oriented. There is a two fold covering $E_T$ of $M$ (that we call the \textit{orthonormal transverse frame bundle}) with the foliation $\mathcal{F}_T$ by the pull-back of $\mathcal{F}$ to this covering. Then we see that $\mathcal{F}_T$ is constructed by a Riemannian submersion $\pi_T:E_T\to E_T/\mathcal{F}_T$, where $E_T/\mathcal{F}_T$ is a circle of radius $\frac{\h}{\pi}$. From the argument of the previous case, we see that $E_T$ is a mapping torus of an isometry $\phi:K\to K$ with a $(n-1)$-dimensional compact connected Riemannian manifold $K$ of nonnegative Ricci curvature and $\mathcal{F}_T$ is formed by $\{L'_t=\{t\}\times K/\sim\}_{t\in [0,2\h]}$. There is an involution $\sigma$ on $E_T$ that induces an involution $\overline{\sigma}$ on $S^1\subset \mathbb{C}^1$ (by conjugations), and hence $M=E_T/\langle\sigma\rangle$ and $\mathcal{F}=\mathcal{F}_T/\langle\sigma\rangle$.

Therefore, we have the conclusion.

\end{proof}

\begin{example}

An example of the first case is $M=([0,2\pi]\times T^2)/\sim$ with $$(0,u,v)\sim (2\pi,-u,-v).$$ Here, we regard $T^2=S^1\times S^1\subset \mathbb{C}^2$. The leaves of $\mathcal{F}$ are of the forms $L_t=\{t\}\times T^2$,  which are $2$-tori. The first basic eigenvalue of $(M,\mathcal{F})$ is $\lambda_1^B=1=\pi^2/\pi^2$ with the corresponding basic eigenfunction $f(t,u,v)=\cos(t)$.
\end{example}
\begin{example}
     An example of the second case is $M=E_T/\langle\sigma\rangle$ where $E_T=([0,2\pi]\times T^2)/\sim$ with $(0,u,v)\sim(2\pi,-u,\overline{v})$ and $\sigma:E_T\to E_T$ such that 
     $$[(t,u,v)]\mapsto [(-t+2\pi,u,v)].$$
     
     The leaves of $\mathcal{F}_T$ on $E_T$ are of the forms $L_t=\{t\}\times T^2 $ are $2$-tori. Then the induced Riemannian foliation  $\mathcal{F}=\mathcal{F}_T/\langle\sigma\rangle$ on $M$ consists of $2$-torus leaves and only two Klein bottle leaves. The first basic eigenvalue of $(M,\mathcal{F})$ is $\lambda_1^B=1=\pi^2/\pi^2$ with the corresponding basic eigenfunction $f(t,u,v)=\cos(t)$.

\end{example}

\begin{remark}
    When $\mathcal{F}$ is given as the (connected) fibers of a Riemannian submersion $\pi:M\to M/\mathcal{F}$ with minimal fibers, the basic spectrum of $(M,\mathcal{F})$ is the same as the spectrum of $M/\mathcal{F}$. If we have $\lambda_1^B=\pi^2/\h_{M/\mathcal{F}}^2$ and $\sec(M)\geq 0$, then applying Hang-Wang rigidity 
     can only conclude that $M/\mathcal{F}$ is isometric to $S^1\left(\frac{1}{\sqrt{\lambda_1^B}}\right)$ without knowing the structure of $\mathcal{F}$. By proving Theorem \ref{main2}, not just know $M/\mathcal{F}$, but we also understand the shape of $M$ and the structure of $\mathcal{F}$.
\end{remark}

\section{The case of a Negative Ricci Curvature Lower Bound }\label{casenegative}

In this section, we study a lower bound for the case of singular Riemannian foliations on a compact Riemannian manifold with the assumption of a \textbf{negative Ricci curvature lower bound.} More precisely, let $(M,\mathcal{F})$ be a closed singular Riemannian foliation on a compact $n$- dimensional manifold $M$ with $\ric(M)\geq (n-1)K$, where $K<0$. 

We recall the gradient estimate by Yang (\cite[Lemma 2 and Lemma 5]{yangneg}) in Lemma \ref{lemmatheta} and Lemma \ref{gradientestimatephi} and establish the similar argument to Yang estimates \cite[Theorem 2 and Theorem 1]{yangneg} to prove Theorem \ref{main3}.

Let $u\in C^\infty(M)$ be an eigenfunction with $\Delta u=-\lambda u$, where $\lambda>0$. Similar to the setting in Sect.~ \ref{7777}, we can assume that $u$ is an eigenfunction with respect to $\lambda$ such that
    $$1=\max u>\min u=-k,0<k\leq 1.$$
Then, rescaling by multiplying with a constant $1-\varepsilon$ (where $\varepsilon>0$ is arbitrary) gives us an eigenfunction $u_\varepsilon$ satisfies that:
\begin{align*}
    \left\{\begin{array}{l}
        \Delta u_\varepsilon=-\lambda u_\varepsilon ,  \\
         \max u_\varepsilon=1-\varepsilon, \min u_\varepsilon=-k(1-\varepsilon).
    \end{array}\right.
\end{align*}

Define $\theta_\varepsilon=\sin^{-1} u_\varepsilon$ and $\phi_\varepsilon=\sin^{-1} f_\varepsilon$, where $f$ is given by
\begin{align*}
    f_\varepsilon=\frac{u_\varepsilon-\frac{(1-k)(1-\varepsilon)}{2}}{\frac{1+k}{2}}.
\end{align*}
Observe that $f_\varepsilon$ satisfies the following equation
\begin{align}
    \left\{\begin{array}{l}
         \Delta f_\varepsilon=-\lambda(f_\varepsilon+a), \\
         \max f_\varepsilon=1-\varepsilon,\min f_\varepsilon=-(1-\varepsilon), 
    \end{array}\right.
\end{align}where $0<a=\frac{1-k}{1+k}(1-\varepsilon)<1$.

 Also the ranges of $\theta_\varepsilon
 ,\phi_\varepsilon$ are $\left[-\frac{\pi}{2}+\varepsilon_2,\frac{\pi}{2}-\varepsilon_1\right],\left[-\frac{\pi}{2}+\varepsilon_1,\frac{\pi}{2}-\varepsilon_1\right]$, respectively, where $\varepsilon_1=\sin^{-1}\sqrt{\varepsilon(2-\varepsilon)}$ and $\varepsilon_2=\sin^{-1}(\sqrt{1-k^2(1-\varepsilon)^2})$.

We will use the following gradient estimates of Yang \cite[Lemma 3 and Lemma 4]{yangneg} with the above setting
\begin{lemma}\label{lemmatheta}
    
    \begin{align}
        \vert\nabla\theta_\varepsilon\vert^2\leq \left(\sqrt{\lambda}+\alpha\sqrt{(1-n)K}\cos\theta_\varepsilon\right)^2,
    \end{align}where $\alpha=\frac{1}{2}\max(\sqrt{n-1},\sqrt{2})$.
\end{lemma}
\begin{proof}[Proof of Lemma 1]
    See \cite[Lemma 2]{yangneg}.

    \qed
    
\end{proof}
\begin{lemma}\label{gradientestimatephi}
    \begin{align}
        \vert\nabla\phi_\varepsilon\vert^2\leq \nu(\phi_\varepsilon)=\lambda-(n-1)K+\lambda a\Psi(\phi_\varepsilon), 
    \end{align}where $\Psi$ is an odd function on $\left[-\frac{\pi}{2}+\varepsilon_1,\frac{\pi}{2}-\varepsilon_1\right]$ and $\nu(\phi)=\lambda-(n-1)K+\lambda a\Psi(\phi)$ is the solution of the following equation
    \begin{align}
        \left\{\begin{array}{l}
             \nu+\frac{1}{2}\left(1+\frac{\lambda}{\lambda-(n-1)K}\right)\nu'\cos\phi\sin\phi-\frac{1}{2}\nu''\cos^2\phi=\lambda-(n-1)K+\lambda a\sin\phi , \\
              \nu\left(-\frac{\pi}{2}+\varepsilon_1\right)=\lambda-(n-1)K-\lambda a, \text{ }\nu\left(\frac{\pi}{2}-\varepsilon_1\right)=\lambda-(n-1)K+\lambda a.
        \end{array}\right.
    \end{align}
\end{lemma}
\begin{proof}[Proof of Lemma 2]
    See \cite[Lemma 5]{yangneg}.

    \qed
    
\end{proof}

Now, we prove the following proposition, which is similar to \cite[Theorem 2]{yangneg}
\begin{proposition}\label{propforneg}
    Let $(M,\mathcal{F})$ be a closed singular Riemannian manifold on an $n$-dimensional compact Riemannian manifold $M$ with $\ric(M)\geq (n-1)K$,where $K<0$. Assume that $(M,\mathcal{F})$ has a basic mean curvature. Then the first basic eigenvalue $\lambda_1^B$ satisfies
    \begin{align}\label{propneg}
        \frac{(1-n)K}{\lambda_1^B}\leq\frac{16}{\pi^2C_n^2}\left[\exp\left(\frac{1}{2}C_n\sqrt{(1-n)K\h_{M/\mathcal{F}}^2}\right)-1\right]^2,
    \end{align}where $C_n=\max(\sqrt{n-1},\sqrt{2})$.
\end{proposition}
\begin{proof}
    We start with a normalized basic eigenfunction $u\in C_B^\infty(M)$ corresponding to $\lambda_1^B$ that satisfies $\Delta u=-\lambda_1^B u, -1\leq \min u=-k<\max u\leq 1$.

  Then, we can define the function $\theta_\varepsilon$ as the above. Use Lemma \ref{lemmatheta}, we get
  \begin{align*}
      \vert\nabla\theta_\varepsilon\vert^2\leq \left(\sqrt{\lambda_1^B}+\alpha\sqrt{(1-n)K}\cos\theta_\varepsilon\right)^2\leq\left[\sqrt{\lambda_1^B}+\alpha\sqrt{(1-n)K}\left(\frac{\pi}{2}-\vert\theta_\varepsilon\vert\right)\right]^2,
  \end{align*}which implies that
  \begin{align}\label{gradient1neg}
      \frac{\vert\nabla\theta_\varepsilon\vert}{\sqrt{\lambda_1^B}+\alpha \sqrt{(1-n)K}\left(\frac{\pi}{2}-\vert\theta_\varepsilon\vert\right)}\leq 1.
  \end{align}

 Let $x\in M$ such that $\theta_\varepsilon(x)=-\frac{\pi}{2}+\varepsilon_2$. Then we choose $y$ from the level set $A=\theta_\varepsilon^{-1}(\frac{\pi}{2}-\varepsilon_1)$ (which is compact and is a union of leaves of $\mathcal{F}$) minimizing the distance between $x$ and $A$ and let $\gamma
        $ be the geodesic joining $x$ and $y$, namely, $L(\gamma)=\h(x,y)=\h(x,A)$.

        Then we have $L(\gamma)=\h(x,A)=\h(x,L_y)=\h(L_x,L_y)$.

        Integrate (\ref{gradient1neg}) along $\gamma$ gives us
        \begin{align}\label{verylong estimate}
            \h_{M/\mathcal{F}}&\geq\int_{-\frac{\pi}{2}+\varepsilon_2}^{\frac{\pi}{2}-\varepsilon_1}\frac{\h \theta}{\sqrt{\lambda_1^B}+\alpha\sqrt{(1-n)K}\left(\frac{\pi}{2}-\vert\theta\vert\right)}\nonumber\\&=\int_0^{\frac{\pi}{2}-\varepsilon}\frac{\h\theta }{\sqrt{\lambda_1^B}+\alpha\sqrt{(1-n)K}\left(\frac{\pi}{2}-\theta\right)}+\int_{-\frac{\pi}{2}+\varepsilon_2}^0\frac{\h\theta}{\sqrt{\lambda_1^B}+\alpha\sqrt{(1-n)K}\left(\frac{\pi}{2}+\theta\right)}\nonumber\\
            &=\frac{1}{\alpha\sqrt{-(n-1)K}}\log\frac{1+\alpha\sqrt{\frac{(1-n)K}{\lambda_1^B}}\frac{\pi}{2}}{1+\alpha\sqrt{\frac{(1-n)K}{\lambda_1^B}}\varepsilon_1}+\frac{1}{\alpha\sqrt{(1-n)K}}\log\frac{1+\alpha\sqrt{\frac{(1-n)K}{\lambda_1^B}}\frac{\pi}{2}}{1+\alpha\sqrt{\frac{(1-n)K}{\lambda_1^B}}\varepsilon_2}.
        \end{align}

By letting $\varepsilon\to 0$, then $\varepsilon_1\to 0,\varepsilon_2\to \sqrt{1-k^2}\leq 1$ and (\ref{verylong estimate}) becomes
\begin{align*}
    \h_{M/\mathcal{F}}\geq\frac{1}{\alpha\sqrt{(1-n)K}}\log\left(1+\alpha\sqrt{\frac{(1-n)K}{\lambda_1^B}}\frac{\pi}{2}\right).
\end{align*}

Solve for $\frac{(1-n)K}{\lambda_1^B}$ on the right hand side, we will get (\ref{propneg}).

\qed

\end{proof}

We will now prove Theorem \ref{main3}, and the argument will be similar to Yang's argument in his proof of \cite[Theorem 1]{yangneg}:

\begin{proof}[Proof of Theorem \ref{main3}]
We recall $u$ in Proposition \ref{propforneg} and define $\phi_\varepsilon$ as the above.

Let $z\in M$ such that $\phi_\varepsilon(z)=-\frac{\pi}{2}+\varepsilon_1$. Then we choose $w$ from the level set $B=\phi_\varepsilon^{-1}(\frac{\pi}{2}-\varepsilon_1)$ minimizing the distance between $z$ and $B$ and let $\gamma'
        $ be the geodesic joining $z$ and $w$. Then we have $L(\gamma')=\h(z,B)=\h(z,L_w)=\h(L_z,L_w)$.

From Lemma \ref{gradientestimatephi}, we have the gradient estimate
\begin{align*}
    \vert\nabla\phi_\varepsilon\vert^2\leq\nu(\phi_\varepsilon)=\lambda_1^B-(n-1)K+\lambda_1^Ba\Psi(\phi_\varepsilon),
\end{align*}or equivalently, we have
\begin{align}\label{phiestimate}
    \frac{\vert\nabla\phi_\varepsilon\vert}{\sqrt{\lambda_1^B-(n-1)K+\lambda_1^B a\Psi(\phi_\varepsilon)}}=\frac{\vert\nabla\phi_\varepsilon\vert}{\sqrt{\nu(\phi_\varepsilon)}}\leq 1.
\end{align}

Integrate (\ref{phiestimate}) along $\gamma'$ using the fact that $\Psi$ is odd, we get
       \begin{align}\label{Cauchy}
           \h_{M/\mathcal{F}}&\geq\int_{-\frac{\pi}{2}+\varepsilon_1}^{\frac{\pi}{2}-\varepsilon_1}\frac{\h\phi}{\sqrt{\lambda_1^B-(n-1)K+\lambda_1^Ba\Psi(\phi)}}\nonumber\\&=\int_{0}^{\frac{\pi}{2}-\varepsilon_1}\frac{\h\phi}{\sqrt{\lambda_1^B-(n-1)K+\lambda_1^Ba\Psi(\phi)}}+\int_{-\frac{\pi}{2}+\varepsilon_1}^{0}\frac{\h\phi}{\sqrt{\lambda_1^B-(n-1)K+\lambda_1^Ba\Psi(\phi)}}\nonumber\\
           &=\int_0^{\frac{\pi}{2}-\varepsilon_1}\left(\frac{\h\phi}{\sqrt{\lambda_1^B-(n-1)K+\lambda_1^Ba\Psi(\phi)}}+\frac{\h\phi}{\sqrt{\lambda_1^B-(n-1)K-\lambda_1^Ba\Psi(\phi)}}\right).
       \end{align}

       Continue estimating (\ref{Cauchy}) using Cauchy's inequality, we get
       \begin{align*}
           \h_{M/\mathcal{F}}&\geq\int_0^{\frac{\pi}{2}-\varepsilon_1}\frac{2\h\phi}{\sqrt[4]{(\lambda_1^B-(n-1)K)^2-(\lambda_1^B)^2a^2\Psi^2(\phi)}}\nonumber\\&\geq\int_0^{\frac{\pi}{2}-\varepsilon_1}\frac{2\h\phi}{\sqrt{\lambda_1^B-(n-1)K}}\nonumber\\&\geq\frac{\pi-2\varepsilon_1}{\sqrt{\lambda_1^B-(n-1)K}}.
       \end{align*}

       Letting $\varepsilon\to 0$, we get 
       \begin{align*}
           \h_{M/\mathcal{F}}\geq\frac{\pi}{\sqrt{\lambda_1^B-(n-1)K}},
       \end{align*}and equivalently, we have
       \begin{align}\label{lastestimate}
           \lambda_1^B\geq\frac{\pi^2}{\h_{M/\mathcal{F}}^2}\frac{1}{1-\frac{(n-1)K}{\lambda_1^B}}.
       \end{align}

Since $C_n\geq\sqrt{2}$, by Proposition \ref{propforneg}, we have
\begin{align}\label{almost}
    \frac{(1-n)K}{\lambda_1^B}&\leq\frac{16}{\pi^2C_n^2}\left[\exp\left(\frac{1}{2}C_n\sqrt{(1-n)K\h_{M/\mathcal{F}}^2}\right)-1\right]^2\nonumber\\&\leq\frac{8}{\pi^2}\left[\exp\left(C_n\sqrt{(1-n)K\h_{M/\mathcal{F}}^2}\right)-2\exp\left(\frac{1}{2}C_n\sqrt{(1-n)K\h_{M/\mathcal{F}}^2}\right)+1\right]\nonumber\\&\leq \exp\left(C_n\sqrt{(1-n)K\h_{M/\mathcal{F}}^2}\right)-1.
\end{align}

Then, substitute (\ref{almost}) into (\ref{lastestimate}), we get
\begin{align*}
    \lambda_1^B\geq\frac{\pi^2}{\h_{M/\mathcal{F}}^2}\exp\left(-C_n\sqrt{(1-n)K\h_{M/\mathcal{F}}^2}\right).
\end{align*}
\qed

\end{proof}

\begin{example}
    Let $M=S^1\times\Sigma_g$ be the product of the circle $S^1$ and $\Sigma_g$ is the surface of genus $g\geq 2$ with a given metric of constant curvature $-1$.

    Then $\ric(M)\geq -1$. Consider the $\mathcal{F}$ of the form $\{\{p\}\times \Sigma_g\vert p\in S^1\}$. This is a simple foliation which is induced from the projection $\pi: S^1\times \Sigma_g\to S^1$, and hence, it has a basic mean curvature.

    Obviously, $\lambda_1^B=1$ with the corresponding basic eigenfunction $f(p,q)=\sin(p)$ and $M/\mathcal{F}=S^1$, which implies $\h_{M/\mathcal{F}}=\pi$.

    Then we see 
    \begin{align*}
        \lambda_1^B=1>\frac{\pi^2}{\h_{M/\mathcal{F}}^2}\exp\left(-C_3\sqrt{(1-3)(-1)\h_{M/\mathcal{F}}^2}\right)=e^{-2\pi},
    \end{align*}which verifies Theorem \ref{main3}.
\end{example}

\section{The case of Positive Ricci Curvature Lower Bound}\label{1111}

In this section, we study a lower bound for the case of singular Riemannian foliations on a compact Riemannian manifold with the assumption of\textbf{ positive Ricci curvature. }. 

Let $(M,\mathcal{F})$ be a closed singular Riemannian foliation on an $n$-dimensional compact Riemannian manifold $M$ with $\ric(M)\geq (n-1)K>0$.

We will follow the ideas that are showed in the work of Bakry-Qian (see  \cite[Theorem 14]{bakryqian}), and the work of Shi-Zhang \cite{shizhang} to prove Theorem \ref{main}.

\begin{lemma}\label{jjj}

For given $K>0,n>1,a,\delta<\frac{\pi}{\sqrt{K}}$, let $\lambda(K,n,\delta,a)$ be the first non-zero eigenvalue of the Neumann problem on the interval of length $\delta$, $[a,a+\delta]$, of the equation given by
\begin{align*}
   L_{K,n}v=-\lambda v\text{ on }(a,a+\delta),
\end{align*} and\begin{align*}
v'(a)=v'(a+\delta)=0,
\end{align*} where $(L_{K,n}v)(x):= v''(x)-(n-1)\sqrt{K}\tan(\sqrt{K}x)v'(x)$

Then we have two properties of $\lambda(K,n,\delta,a)$:
\begin{enumerate}[(1)]
    \item $\lambda(K,n,\delta,a)\geq \lambda(K,n,\delta,-\delta/2)$, in other words, the central interval has the lowest Neumann eigenvalue.
    \item The function $\h\mapsto \lambda(K,n,\h,-\h/2)$ is a non-increasing function.
\end{enumerate}

\end{lemma}

\begin{proof}
    To prove (1), see \cite[Theorem 13]{bakryqian}.

    For proving (2), see the discussion of eigenvalues on symmetric intervals \cite[Section 3]{bakryqian}, and the last argument in \cite[Theorem 14]{bakryqian}.

\end{proof}

In the following lemma, we proceed exactly the gradient estimate argument in the proof of \cite[Theorem 14]{bakryqian}, and  choose a horizontal geodesic to take an integral, similar to the way that is introduced in the proof of Theorem \ref{main3} and Theorem \ref{main1}.

\begin{lemma}\label{999}
    Let $M^n$ be a compact Riemannian manifold and $\mathcal{F}$ be a singular Riemannian foliation of $M$ with closed leaves and basic mean curvature. Assume the Ricci curvature of $M$ is bounded below by $(n-1)K>0$. Suppose $\lambda_1^B$ is the first basic eigenvalue of $(M,\mathcal{F})$ and $\h_{M/\mathcal{F}}$ is the diameter of the leaf space $M/\mathcal{F}$. Then $$\lambda_1^B\geq \lambda(K,n,\h_{M/\mathcal{F}},-\h_{M/\mathcal{F}}/2).$$
\end{lemma}
\begin{proof}
   First, let $u$ be the normalized (basic) eigenfunction with the eigenvalue $\lambda_1^B$ such that $\min u=-1$ and $0<\rho=\max u\leq 1$. 

According to \cite[Theorem 9]{bakryqian}, we have $\rho\geq w(b)$ where $b=\inf\left\{x>-\frac{-\pi}{2}\vert w'(x)=0\right\}$ and $w$ satisfies
\begin{align*}
    \left\{\begin{array}{l}L_{K,n}w=-\lambda_1^Bw, \\w\left(\frac{-\pi}{2}\right)=-1,w'\left(-\frac{\pi}{2}\right)=0.
    \end{array}\right.
\end{align*}

Then, from \cite[Corollary 1]{bakryqian}, we can choose a subinterval $[a,b]\subset \left[-\frac{\pi}{2},\frac{\pi}{2}\right]$ such that there exists a solution $v$ of the equation
\begin{align*}
    \left\{\begin{array}{l}
         L_{K,n}v=-\lambda_1^Bv,  \\
         v(a)=-1,v'(a)=0,\\
         v(b)=\rho,v'(b)=0.
    \end{array}\right.
\end{align*}

  Then, apply \cite[Theorem 8]{bakryqian} for the function $g=v^{-1}\circ u$ to get that  $\vert \nabla g\vert\leq 1$. 
   
   Let $x\in M$ such that $u(x)=-1$. Choose $y$ from the level set $A=u^{-1}(\rho)$ and let $\gamma
        $ be the geodesic from $x$ to $y$ that minimizes the distance from $x$ to the $ u^{-1}(\rho)$. That gives us
        \begin{align*}
            L(\gamma)=\h(x,y)=\h(x,A).
        \end{align*}

        This geodesic $\gamma$ is also the shortest geodesic from $x$ to $L_y\subset A$, which means $\h(x,A)=\h(x,L_y)$.

        The global equidistance of $\mathcal{F}$ implies  that $\h(x,L_y)=\h(L_x,L_y)$ and hence
        \begin{align*}
            \h(x,y)=L(\gamma)=\h(L_x,L_y)\leq\sup_{p,q\in M}\h(L_p,L_q)=\h_{M/\mathcal{F}}.
        \end{align*}

        Since $\vert\nabla g\vert\leq 1$, we get that\begin{align}\label{kkk}
            \h_{M/\mathcal{F}}\geq\h(x,y)=\int_\gamma \h t\geq\left\vert\int_\gamma\nabla g\cdot \h r\right\vert=\vert g(y)-g(x)\vert=b-a.
        \end{align}

        From Lemma \ref{jjj}, we see that \begin{align*}
            \lambda_1^B=\lambda(K,n,b-a,a)\geq \lambda(K,n,b-a,-(b-a)/2).
        \end{align*}

       Because the function $\h\mapsto\lambda(K,n,\h,-\h/2)$ is non-increasing, hence, use (\ref{kkk}),  we get that
        \begin{align*}
            \lambda_1^B\geq\lambda(K,n,\h_{M/\mathcal{F}},-\h_{M/\mathcal{F}}/2).
        \end{align*}
\end{proof}

Finally, we are ready to prove Theorem \ref{main}:
\begin{proof}[Proof of Theorem \ref{main}]
First, assume that $K>0$.

    Use Lemma \ref{999}, then we get that $\lambda_1^B\geq \lambda(K,n,\h_{M/\mathcal{F}},-\h_{M/\mathcal{F}}/2)$. Let $v$ be the corresponding eigenfunction of $\lambda(K,n,\h_{M/\mathcal{F}},-\h_{M/\mathcal{F}}/2)$ of the equation
    \begin{align*}
        L_{K,n}v=-\lambda v\text{ on }\left(-\frac{\h_{M/\mathcal{F}}}{2},\frac{\h_{M/\mathcal{F}}}{2}\right),
    \end{align*}and $v'\left(-\frac{\h_{M/\mathcal{F}}}{2}\right)=v'\left(\frac{\h_{M/\mathcal{F}}}{2}\right)=0$.

    Notice that $v$ is odd, so $v(0)=0$. Without loss of generality, we assume that $v(x)> 0$ if $x>0$.

    Recall the most important calculation from Shi-Zhang (see Lemma 2.2, \cite{shizhang})
    \begin{lemma}[Lemma 2.2 of \cite{shizhang}]\label{abcde}
        Assume $\lambda$ and $f(x),0\leq x\leq l$ satisfy that
        \begin{align*}
            f''(x)+F(x)f'(x)=-\lambda f(x),\text{ }x\in [0,l],
        \end{align*}where $F(x)$ be a differentiable function on $[0,l]$ with $F(0)=0$ and $f(x)$ satisfies the conditons $f(0)=0, f(x)>0$ when $x>0$, and $f'(l)=0$.

        Denote $F_0:=\max \{F'(x):0\leq x\leq l\}$, then we have the following estimate
        \begin{align}\label{zzz}
            \lambda\geq-s^2\left(\frac{\pi}{l}\right)^2+s\left(\left(\frac{\pi}{l}\right)^2-F_0\right).
        \end{align}for any constant $s$ such that $0<s<1$.
    \end{lemma}
    \begin{proof}[Proof of Lemma 4.4] See  \cite[Lemma 2.2]{shizhang}.

    \end{proof}
To finish the proof of Theorem \ref{main}, we consider two cases:

 \textbf{Case 1:} If $\h_{M/\mathcal{F}}=\frac{\pi}{\sqrt{K}}$, then the RHS of (\ref{uuuuuu}) can be written as
\begin{align*}
    4s(1-s)K+s(n-1)K=[4s(1-s)+s(n-1)]K\leq [1+(n-1)K]=nK.
\end{align*}

The inequality (\ref{uuuuuu}) is true, indeed, by following Lichnerowicz's estimate, we see that $\lambda_1^B\geq\lambda_1\geq nK$.

\textbf{Case 2:} If $\h_{M/\mathcal{F}}<\frac{\pi}{\sqrt{K}}$, apply Lemma \ref{abcde} for the function $F(x)=-(n-1)\sqrt{K}\tan(\sqrt{K}x)$, the eigenfunction $f(x)=v(x)$, $l=\h_{M/\mathcal{F}}/2$, and $\lambda=\lambda(K,n,\h_{M/\mathcal{F}},-\h_{M/\mathcal{F}}/2)$ with an observation that
\begin{align*}
    F_0&=\max\{F'(x):0\leq x \leq\h_{M/\mathcal{F}}\}\\&=\max\{-(n-1)K\sec^2(\sqrt{K}x):0\leq x \leq\h_{M/\mathcal{F}}\}\\&=-(n-1)K,
\end{align*} and from (\ref{zzz}), we have that
\begin{align*}
    \lambda_1^B&\geq\lambda(K,n,\h_{M/\mathcal{F}},-\h_{M/\mathcal{F}}/2)\\&\geq -s^2\left(\frac{2\pi}{\h_{M/\mathcal{F}}}\right)^2+s\left(\left(\frac{2\pi}{\h_{M/\mathcal{F}}}\right)^2+(n-1)K\right)\\&=4s(1-s)\frac{\pi^2}{\h_{M/\mathcal{F}}^2}+s(n-1)K.
\end{align*}

Therefore, the inequality (\ref{uuuuuu}) holds when $K>0$.

In the case $K=0$, notice that for every $s$ such that $0<s<1$
\begin{align*}
    4s(1-s)\frac{\pi^2}{\h_{M/\mathcal{F}}^2}+s(n-1)K&=4s(1-s)\frac{\pi^2}{\h_{M/\mathcal{F}}^2}\\&\leq\sup_{s\in (0,1)}4s(1-s)\frac{\pi^2}{\h_{M/\mathcal{F}}^2}=\frac{\pi^2}{\h_{M/\mathcal{F}}^2}.
\end{align*}

Recall Theorem \ref{main1}, then we confirm that the inequality (\ref{uuuuuu}) is true when $K= 0$.

\end{proof}

\begin{example}
    Let $S^1=\{e^{i\theta}\vert \theta\in [0,2\pi)\}$ act isometrically on the round sphere  $S^{2n+1}\subseteq \mathbb{C}^{n+1}$ by the action $e^{i\theta}\cdot (z_1,\cdots,z_{n+1})=(e^{i\theta}z_1,\cdots,e^{i\theta}z_{n+1})$. This action defines a homogeneous foliation $\mathcal{F}$ on $S^{2n+1}$ by its $S^1$-orbits and the leaf space is $S^{2n+1}/\mathcal{F}=\mathbb{CP}^n$, an $2n$-dimensional complex projective space with the Fubini-Study metric.

    We see that the first basic eigenvalue is $\lambda_1^B=2(2n+2)$ with the corresponding basic eigenfunction is the restriction on $S^{2n+1}$ of a $S^1-$invariant harmonic homogenous polynomial of degree $2$ $$f(x_1,\cdots,x_{2n+2})=f(z_1,\cdots,z_{n+1})=\vert z_1\vert^2-\vert z_2\vert^2=x_1^2+x_{n+2}^2-x_2^2-x_{n+3}^2,$$where $z_j=x_j+ix_{n+1+j}$ for all $j\in \{1,\cdots,n+1\}$.

Also, we have $\h_{\mathbb{CP}^n}=\pi/2$. We verify Theorem \ref{main}
    \begin{align*}
        \lambda_1^B\geq 4s(1-s)\frac{\pi^2}{\h_{\mathbb{CP}^n}^2}+s(2n+1-1),
    \end{align*}for all $0<s<1$ with the choice of $K=1$.

    The inequality is equivalent to
    \begin{align*}
        2(2n+2)\geq16s(1-s)+s(2n+1-1).
    \end{align*}Rearranging terms to get the equivalent inequality
    \begin{align*}
        2(2-s)n+4[1-4s(1-s)]\geq 0.
    \end{align*}

    This is true for all $s\in (0,1)$ and we see Theorem \ref{main} holds in this case. 
\end{example}

\begin{example}
    \textit{Isoparametric foliations} on spheres are well-known examples of singular Riemannian foliations, where the reader can find a long history with contributions by many mathematicians, see \cite{miguel} and \cite{chi} for more detail. 

    Such a singular Riemannian foliation $(S^n,\mathcal{
    F})$
    can be determined by an \textit{isoparametric submanifold} $N\subset S^n$ and taking parallel submanifolds of that $N$ to obtain an isoparametric foliation on that sphere $S^n$. Let $g$ be the number of principal curvatures of $N\subset S^n$. Then, in the case when $N$ is of codimension one, M\"unzner proved that $g\in \{1,2,3,4,6\}$ (see Satz A, \cite{munzner1}).  

Also in the proof of Theorem 2 of M\"unzner's paper (see \cite[Satz 2]{munzner1}), the (homogeneous) polynomial $G=g(g+n-1)F-cr^g$ of degree $g$ is harmonic, where $c$ is chosen ($c=0$ if $g$ is odd),  $r^2=x_0^2+\cdots +x_{n}^2$ and $F:\mathbb{R}^{n+1}\to\mathbb{R}$ is the Cartan-M\"unzner (homogeneous) polynomial of degree $g$ that satisfies $\Delta F=cr^{g-2}$ and $\Vert\nabla F\Vert^2=g^2r^{2g-2}$ ($\nabla$ and $\Delta$ are gradient and Laplacian on $\mathbb{R}^{n+1}$). Since the restriction of $G$ on $S^n$ is basic, $G\vert_{S^n}$ is a basic eigenfunction of $(S^n,\mathcal{F})$ and hence, $g(g+n-1)$ is a basic eigenvalue of $(S^n,\mathcal{F})$. Furthermore, the algebra of basic (homogeneous) polynomials is finitely generated by two generators $\rho_1=r^2$ and $\rho_2=F$. Hence the lowest degree of a basic homogeneous polynomial which restricts to a non-constant function on the sphere is $g$. This implies that $\lambda_1^B=g(g+n-1)$ is the first non-zero basic eigenvalue of $(S^n,\mathcal{F})$.

    The diameter of the leaf space $S^n/\mathcal{F}$ is exactly the distance between two singular leaves, which are level sets $V^{-1}(-1)$ and $V^{-1}(1)$ of the map $V=F\vert_{S^n}$, 
    and hence we get $\h_{S^n/\mathcal{F}}=\pi/g$ (see \cite[Satz 4]{munzner1}). 
    
    Now, we choose $K=1$ and verify Theorem \ref{main} is true.

    Indeed, the inequality (\ref{uuuuuu}) is equivalent to
    \begin{align*}
        g(n+g-1)\geq 4s(1-s)g^2+s(n-1),
    \end{align*}and hence
    \begin{align*}
        [1-4s(1-s)]g^2+(g-s)(n-1)\geq 0.
    \end{align*}

    Since $g\geq 1$ and with a fixed $s\in (0,1)$, the function $f(x)=[1-4s(1-s)]x^2+(x-s)(n-1)$ is increasing, then $f(g)\geq f(1)=1-4s(1-s)+(1-s)(n-1)\geq 0$.

    Therefore, we confirm that Theorem \ref{main} is true in this case.
\end{example}
\begin{remark}
    Theorem \ref{main} can be applied particularly for $s=\frac{1}{2}$ to obtain Li's type estimate:
    \begin{align}\label{litype}
        \lambda_1^B\geq\frac{\pi^2}{\h_{M/\mathcal{F}}^2}+\frac{1}{2}(n-1)K
    \end{align}

    We can actually learn more about which $s$ will give the best estimate: set $A=\frac{\pi^2}{\h_{M/\mathcal{F}}^2},B=(n-1)K$, then observe that the right hand side of the inequality (\ref{uuuuuu}) is a quadratic function of $s$:
    \begin{align*}
        b(s)=4s(1-s)A+sB=-4As^2+(4A+B)s.
    \end{align*}

    The only zero of $b'(s)$ is $s=s_0=\frac{4A+B}{8A}$.

    Therefore, by studying the function $b(s)$ on $(0,1)$, we can see that:
    \begin{itemize}
        \item If $s_0\leq 1$, which is equivalent to
        \begin{align*}
            \h_{M/\mathcal{F}}\leq\frac{2}{\sqrt{(n-1)}}\frac{\pi}{\sqrt{K}}
        \end{align*}(notice that this is always the case when $n\leq 5$), then the maximum of $b(s)$ is \begin{align*}
            b(s_0)=\frac{(4A+B)^2}{16A}=A+\frac{B}{2}+\frac{B^2}{16A}=\frac{\pi^2}{\h^2_{M/\mathcal{F}}}+\frac{1}{2}(n-1)K+\frac{B^2}{16A},
        \end{align*} and it is sharper than the Li's type estimate (\ref{litype}).
        \item If $s_0> 1$, which is equivalent to
        \begin{align*}
            \frac{2}{\sqrt{n-1}}\frac{\pi}{\sqrt{K}}<\h_{M/\mathcal{F}}\leq\frac{\pi}{\sqrt{K}}
        \end{align*} (it is possible when $n>5$), then the supremum of $b(s)$ is by letting $s\to 1^-$, and we get $\lambda_1^B\geq (n-1)K$, which is weaker than Lichnerowicz's estimate. 
    \end{itemize}

    So, we can see that Theorem \ref{main} is a sharp estimate when considering a small enough diameter of the leaf space, or even better, a small diameter of $M$. More precisely, when $\h_{M/\mathcal{F}}\sqrt{K}\leq 2\pi/\sqrt{(n-1)} $, we have the optimal one  
    \begin{align*}
        \lambda_1^B\geq\frac{\pi^2}{\h_{M/\mathcal{F}}^2}+\frac{1}{2}(n-1)K+\frac{(n-1)^2K^2\h_{M/\mathcal{F}}^2}{16\pi^2}.
    \end{align*}
\end{remark}

\printbibliography

\end{document}